\documentclass[11pt,letterpaper]{article}

\usepackage{amsmath}
\usepackage{amsthm}
\usepackage{amsfonts}
\usepackage{amssymb}
\usepackage{enumerate}
\usepackage{setspace}
\usepackage{comment}
\usepackage[percent]{overpic}
\usepackage{graphicx}
\usepackage{float}
\usepackage{tikz}
\usepackage{calc}
\usepackage{pgfplots}
\usepackage{subfigure}
\usepackage{changepage}
\usepackage{array}
\newcolumntype{L}[1]{>{\raggedright\let\newline\\\arraybackslash\hspace{0pt}}m{#1}}
\newcolumntype{C}[1]{>{\centering\let\newline\\\arraybackslash\hspace{0pt}}m{#1}}
\newcolumntype{R}[1]{>{\raggedleft\let\newline\\\arraybackslash\hspace{0pt}}m{#1}}

\newtheorem{theorem}{Theorem}[section]

\newtheorem{corollary}[theorem]{Corollary}
\theoremstyle{definition}

\theoremstyle{remark}

\newcommand{\Rel}[1]{\text{Rel}(#1,q)}

\begin{document}
\title{All Terminal Reliability Roots of Smallest Modulus}
\author{Jason I. Brown\thanks{Contributing author. Supported by NSERC grant 170450-2013}\\
\small Department of Mathematics \& Statistics\\[-0.8ex]
\small Dalhousie University\\[-0.8ex] 
\small Halifax, CA\\
\small\tt Jason.Brown@dal.ca\\
\and
Corey D. C. DeGagn\'e\\
\small Department of Mathematics \& Statistics\\[-0.8ex]
\small Dalhousie University\\[-0.8ex]
\small Halifax, CA\\
\small\tt degagne@dal.ca
}

\maketitle

\begin{abstract}
\noindent Given a connected graph $G$ whose vertices are perfectly reliable and whose edges each fail independently with probability $q\in[0,1],$ the \textit{(all-terminal) reliability} of $G$ is the probability that the resulting subgraph of operational edges contains a spanning tree (this probability is always a polynomial in $q$). The location of the roots of reliability polynomials has been well studied, with particular interest in finding those with the largest moduli.  In this paper, we will discuss a related problem -- among all reliability polynomials of graphs on $n$ vertices, which has a root of  smallest modulus? We prove that, provided $n \geq 3$, the roots of smallest moduli occur precisely for the cycle graph $C_n$, and the root is unique.\\


\end{abstract}

\setstretch{1.4}

\section{Introduction}
A well known model of network robustness is the {\em all-terminal reliability} (or simply {\em reliability}) of a finite undirected graph $G$ (possibly with loops and/or multiple edges), in which the vertices are always operational, but each edge fails independently with probability $q \in [0,1]$ (or equivalently, independently operate with probability $p = 1-q$). The reliability of $G$, $\Rel{G}$, is the probability that the operational edges form a connected spanning subgraph, that is, that the operational edges contain a spanning tree. It is easy to see that the reliability of a graph $G$ is always a polynomial in $q$ and is not identically $0$ if and only if $G$ is connected. For example, if $G$ is a tree of order $n$ (that is, with $n$ vertices) then the reliability is $(1-q)^{n-1}$, as all $n-1$ edges must be operational, and if $G$ is a cycle of order $n$, then its reliability is $(1-q)^{n} + nq(1-q)^{n-1}$, as the operational edges contain a spanning tree if and only if at most one edge has failed. 

Much of the work on reliability has focused on efficient ways of estimation (see, for example, \cite{colbook}), but analytic properties of the functions has also attracted considerable attention as well \cite{browncolbourn,brownkoc,changshrock,graves,moore,wagner}.
As the reliability of a connected graph is always a polynomial in $q$, it is natural to consider the location of the roots ({\em reliability roots})  in the complex plane (when we speak about reliability roots we always assume that the graph in question is connected). It was conjectured originally that the roots lie in the unit disk \cite{browncolbourn}, but some roots were eventually found of modulus greater than $1$ \cite{false} (but only barely -- the furthest away a root has been found from the origin is only approximately $1.113$ \cite{brownlucas}). 

If $M_{n}$ denotes the maximum modulus of a reliability root of a connected graph of order $n$, then all that is known is that $M_{n} \leq n-1$ \cite{brownlucas}, and that for sufficiently large $n$, $M_{n}>1$ \cite{false}. Moreover, as all real reliability roots are in $[-1,0) \cup \{1\}$ \cite{browncolbourn}, we see that for large enough $n$, a reliability root of largest modulus will not be real, and there is no conjecture as to what it might be.

\begin{figure}[h]
\begin{center}
\label{rel8}
\includegraphics[height=100mm]{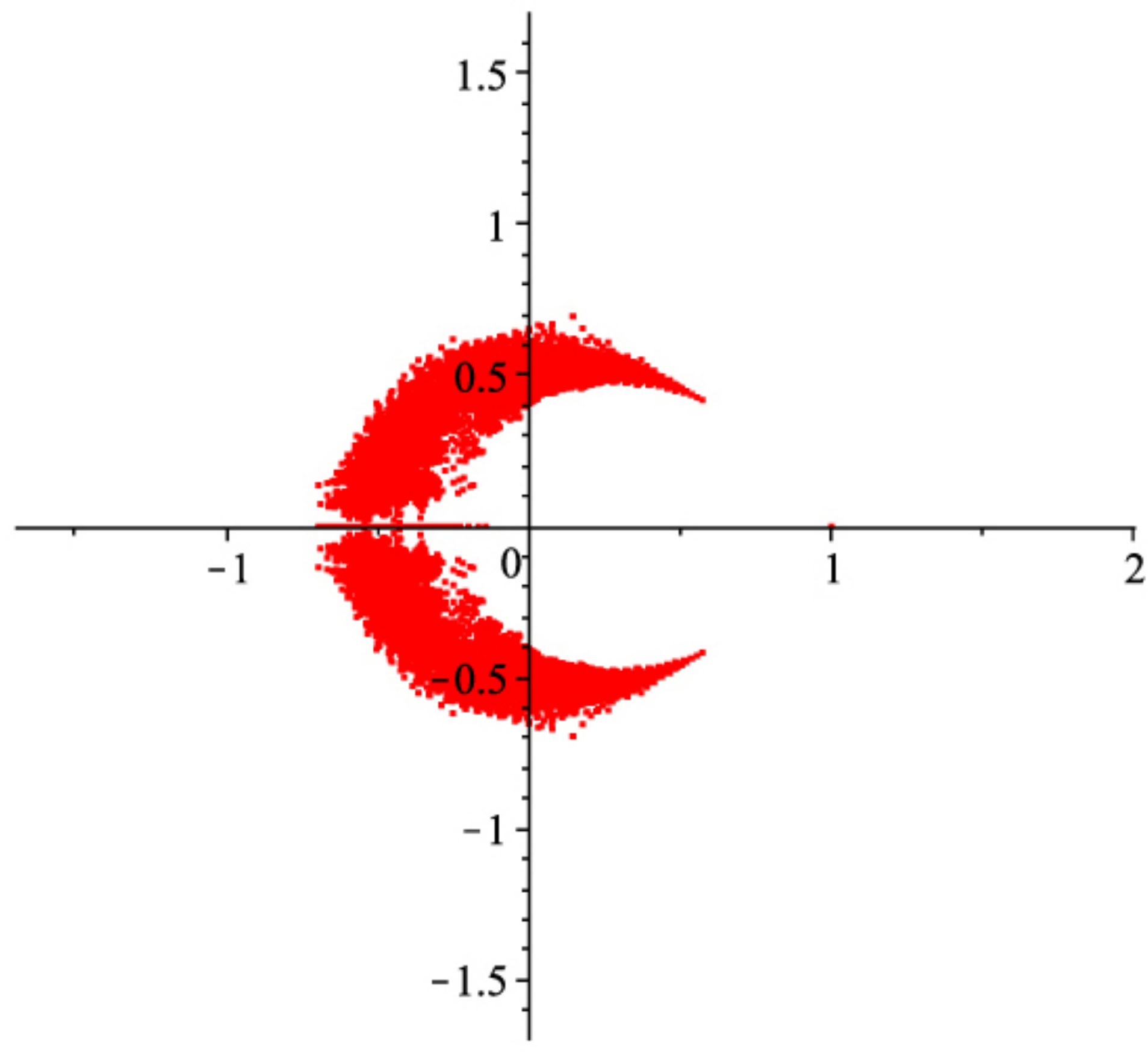}
\caption{Reliability roots of all simple graphs of order $8$.}
\end{center}
\end{figure}

However, what can be said about $m_{n}$, the {\em minimum} modulus of a reliability root of a connected graph of order $n$? Is it as seemingly intractable as its larger counterpart, $M_{n}$? We will completely determine, for all $n$, not only $m_{n}$ but find all roots of this modulus and all graphs that have a root that attains this modulus.

\vspace{0.25in}

We shall assume the reader is acquainted with the basics of graphs and matroid theory (see, for example, \cite{bonmurt} and \cite{oxley}). All graphs under consideration will be assumed to be connected.

\section{Reliability Roots of Smallest Moduli}

We can expand the reliability of a graph $G$ of order $n$ and size $m$ (that is, with $m$ edges) in terms of different bases (see, for example, \cite{colbook}). Two useful expansions are as follows:

\begin{eqnarray}
\Rel{G} & = & \sum_{i = 0}^{m-n+1} F_iq^i(1-q)^{m-i}~~~~~(\mbox{F-Form}) \label{Fform}\\
             & = & (1-q)^{n-1}\sum_{i=0}^{m-n+1}H_iq^i~~~~(\mbox{H-Form})  \label{Hform}
\end{eqnarray}

\noindent Each $F_{i}$ counts the number of subsets of $i$ edges whose deletion leaves $G$ still connected; the collection of such subsets is a (simplicial) complex known as the {\em cographic matroid} of $G$, $\mbox{Cog}(G)$ (the members of $\mbox{Cog}(G)$ are called its {\em faces}). The {\em dimension} of the complex $d$, is the common cardinality of any maximal set, and when the graph is loopless (as we shall assume unless otherwise mentioned, as loops do not affect the reliability), it is $d = m-n+1$, the {\em corank} of graph $G$. 

What $H_{i}$ represents is not so clear from (\ref{Hform}), but has a number of interesting and useful interpretations (see, for example \cite{colbook}):
\begin{itemize}
\item There is a partition of the faces of $\mbox{Cog}(G)$  into intervals $[\sigma,\tau] = \{ \alpha \in \mbox{Cog}(G): \sigma \subseteq \alpha \subseteq \tau \}$, where $\sigma$ and $\tau$ are faces of $\mbox{Cog}(G)$, $\sigma \subseteq \tau$ and $\tau$ is a {\em maximal} face (with respect to inclusion), which necessarily has cardinality $d$. Then $H_{i}$  counts the number of lower sets $\sigma$ that have cardinality $i$.
\item There is an {\em order ideal of monomials}, that is, a set of monomials ${\mathcal M}(G)$ closed under division, such that $H_{i}$ counts the number of monomials in the set with degree $i$. (The construction of such a set of monomials ${\mathcal M}(G)$  can be achieved through connections to commutative algebra -- see, for example, \cite{billera}.) 
\end{itemize}
    
\vspace{0.25in}

We are now ready to determine $m_{n}$.

\begin{theorem}\label{mainthm}
Let $n \geq 2$.  Then $m_n = \frac{1}{n-1}$, that is, the smallest modulus of a reliability root is $1/(n-1)$.
\end{theorem}
\begin{proof}
First we observe that the reliability of the cycle of order $n \geq 2$, $C_{n}$, is given by 
\[ \Rel{C_{n}} = (1-q)^{n} + nq(1-q)^{n-1} = (1-q)^{n} (1+(n-1)q),\]
which has a root at $q = -1/(n-1)$, so clearly $m_{n} \leq 1/(n-1)$. 

We next show that there is no graph of order $n$ that has a reliability root of smaller modulus. Note that if $n = 2$, then $G$ is a bundle of edges between two points, and hence has reliability of the form $1 - q^m$. This polynomial has all of its roots on the unit circle, and indeed its roots all have modulus $1 = 1/(n-1)$. Thus when $n = 2$, there is no graph of order $n$ with a reliability root of modulus smaller than $1/(n-1)$. We assume now that $n \geq 3$ and proceed by induction. 

Let $G$ be any graph of order $n$ and size $m$ (possibly with multiple edges), whose reliability is given by (\ref{Fform}) and (\ref{Hform}). Without loss of generality, $G$ is loopless, for if $G$ had any loops, then it would have the same reliability as the graph $G^\prime$ formed from $G$ by removing all the loops, and $G^\prime$ would also be of order $n$. 

If $G$ has a bridge $e$, and $G-e$ has components $G_{1}$ and $G_{2}$ of orders $n_{1}$ and $n_{2}$ respectively (without loss of generality, $n_{1} \leq n_{2}$), then clearly 
\[ \Rel{G} = (1-q) \cdot \Rel{G_{1}} \cdot \Rel{G_{2}}.\]
If $n_{1} = 1$, then $n_{2} \geq 2$ and $\Rel{G} = (1-q)\cdot \Rel{G_{2}}$, and in this case, by induction $\Rel{G_{2}}$ has no root of modulus smaller than $1/(n_{2}-1) > 1/(n-1)$, and hence the same is true for $\Rel{G}$. If $n_{1} \geq 2$, then by induction, $G_{i}$ has no reliability root of modulus smaller than $1/(n_{i}-1) > 1/(n-1)$ for $i \geq 2$, and again $G$ has no reliability root of modulus smaller than $1/(n-1)$. Thus we can assume that $G$ has no bridges, and so $F_{1}$, the number of edges whose removal leaves $G$ connected, is $m$.

We now examine the $H$-form of $\Rel{G}$ in more detail. We need first to determine $H_0$ and $H_1$.  From the connection between the $H_{i}$ and interval partitions of the cographic matroid,  it is easy to see that $H_0 = 1$, as the empty set is always a face (the unique face of cardinality $0$), and hence is the lower set in exactly one interval of an interval partition of $\mbox{Cog}(G)$. As $F_{1} = m$, each of the $m$ edges is a face of $\mbox{Cog}(G)$, but only $d = m-n+1$ of them appear in the interval whose lower set is $\emptyset$ (as the upper sets always have cardinality $d$). It follows that $H_{1} = m-(m-n+1) = n-1$. Moreover, it is easy to see from the connection to order ideals of monomials that for $0 \leq i \leq d-1$, 
\[ \frac{H_i}{H_{i+1}} \geq \frac{H_0}{H_{1}} = \frac{1}{n-1}.\]
This inequality is equivalent to $(n-1)H_{i} \geq H_{i+1}$, which holds as in any associated order ideal of monomials ${\mathcal M}(G)$, multiplying each monomial $m \in {\mathcal M}(G)$ of degree $i$ by each variable $x$ certainly covers all monomials of degree $H_{i+1}$ at least once. (Alternatively, one can make reference to a result from Huh \cite{huh}, where it was shown that the $H$-vector $\langle H_{0},H_{1},\ldots,H_{d}\rangle$ of any representable matroid -- and in particular, any cographic matroid -- is {\em log concave}, that is, for $1 \leq i \leq d-1$,
$$ H_{i-1}H_{i+1}\leq H_i^2.$$
From this it follows that for $1 \leq i \leq d-1$,
$$ \frac{H_{i-1}}{H_i} \leq \frac{H_i}{H_{i+1}},$$
and so
		\begin{equation}
			\label{eqn1}
			\frac{H_0}{H_1} \leq \frac{H_1}{H_2} \leq \dots \leq \frac{H_{d-1}}{H_d}.
		\end{equation}
However, we shall not need the full force of Huh's result.)

		
We now turn to the well known Enestr\"om-Kakeya theorem (see, for example, \cite[pg.\, 255]{rahman}), which states that if a real polynomial $g(x) = a_kx^k + a_{k-1}x^{k-1} + \dots + a_1x + a_0$ has positive coefficients, then all the roots of $g$ lie in the annulus $r \leq |x| \leq R$ where $r = \min_{0\leq j \leq k-1} \{a_j/a_{j+1}\}$ and $R = \max_{0\leq j \leq k-1}\{ a_j/a_{j+1}\}$.
Consider the {\em h-polynomial}
\[ H(x) = \sum_{i=0}^{d}H_{i}x^{i} \]
of graph $G$, so that 
\[ \Rel{G}=(1-q)^{n-1}H(q).\]
The reliability roots of $G$ are therefore those of $H(q)$, together with $1$, so it suffices to show that $H(q)$ has no root of modulus smaller than $1/(n-1)$. However, from (\ref{eqn1}), the minimum value of the ratios of successive $H_{i}$'s is
\[ r = \frac{H_{0}}{H_{1}} = \frac{1}{n-1}.\]
We deduce from the Enestr\"om-Kakeya Theorem that the reliability polynomial of $G$ has no root with modulus less than $1/(n-1)$, and we are done.
\end{proof}

We can indeed say more about the reliability roots of smallest modulus. For $n = 1$, there are no reliability roots, as the only reliability polynomial is $1$, which has no roots.  For $n = 2$, we have seen that the reliability polynomial has the form $1 - q^m$, and hence all the $m$-th roots of unity are reliability roots of smallest modulus. The situation is much different for all larger orders.

\begin{theorem}
For $n \geq 3$, the only reliability root of minimum modulus $m_{n} = 1/(n-1)$ is $-1/(n-1)$, and only occurs for the cycle $C_{n}$.
\end{theorem}
\begin{proof}
A result of \cite{ASV1} states that a polynomial $g(x) = a_kx^k + a_{k-1}x^{k-1} + \dots + a_1x + a_0$ has a root of modulus $r = \min_{0\leq j \leq k-1} \{a_j/a_{j+1}\}$ only if 
\[ \mbox{gcd}\left(\left\{i \in \{1,2,\ldots,k\} : \frac{a_{i-1}}{a_{i}}  > r \right\}\right) > 1.\] 
Thus we need to consider when 
\[ \frac{H_{i-1}}{H_{i}} > \frac{H_{0}}{H_{1}} = \frac{1}{n-1},\]
that is, when
\begin{eqnarray} 
(n-1)H_{i-1} & > & H_{i}. \label{varh}
\end{eqnarray}
Clearly this fails when $i = 1$, as $H_{0} = 1$ and $H_{1} = n-1$.
We show now that (\ref{varh}) holds for all $i =2,3,\ldots,d$.
 
From the interpretation of $H_{i}$'s as counting the number of monomials of degree $i$ in an order ideal of monomials ${\mathcal M} = {\mathcal M}_{G}$, we see that ${\mathcal M}$ contains $n-1$ variables (i.e. monomials of degree $1$). If for some $i \geq 2$ (\ref{varh}) holds, then if $x$ is any variable and $m$ any monomial of degree $i-1$ in ${\mathcal M}$, $xm$ must be a monomial (of degree $i$) in ${\mathcal M}$ and every monomial of degree $i$ in ${\mathcal M}$ must arise uniquely in this way. As $n \geq 3$, ${\mathcal M}$ has at least $n-1 \geq 2$ variables. It follows that some monomial of degree $i$ in ${\mathcal M}$ must have the form $xym^\prime$ for some monomial $m^\prime$ of degree $i-2$ (if some monomial of degree $i-1$ contains two distinct variables, add any variable to it, and if $x^{i-1}$ is a monomial of degree $i-1$ in ${\mathcal M}$, then for any other variable $y$, $yx^{i-1}$ must be a monomial of degree $i$ in ${\mathcal M}$). However, then the monomial $xym^\prime$ arises by adding variable $x$ to the monomial $ym^\prime$ of degree $i-1$ in ${\mathcal M}$, while $xym^\prime$ arises also by adding variable $y$ to the monomial $xm^\prime$ of degree $i-1$ in ${\mathcal M}$. This contradicts the fact that every monomial of degree $i$ in ${\mathcal M}$ arises uniquely by adding a variable to a monomial of degree $i-1$ in ${\mathcal M}$. 

It follows that 
\[ \mbox{gcd}\left(\left\{i \in \{1,2,\ldots,d\}) : \frac{H_{i-1}}{H_{i}}  > r  = \frac{H_{0}}{H_{1}}\right\}\right) = \mbox{gcd}(\{2,3,\ldots,d = m-n+1\}),\]
and so, if $d = m-n+1$, the corank of $G$, is at least $3$, then the gcd of the set is $1$, and we conclude that there is no root of modulus $r = 1/(n-1)$. If the corank of $G$ is $0$ then $G$ is a tree, and its only reliability root is $1$, which is greater than $1/(n-1)$. If the corank of $G$ is $1$, then provided $G$ is unicyclic, say with a cycle of length $k \leq n$, and its reliability is $(1-q)^{n-k}\cdot \Rel{C_{k}}.$ Thus the reliability roots of $G$ are $1$ and $-1/(k-1)$, and hence there is a root of modulus $1/(n-1)$ if and only if $n = k$ and $G = C_{n}$ (and in this case, the only such root is $-1/(n-1)$). 

All that remains is the case where $G$ is of corank $2$, that is, when $m = n+1$. As in the proof of Theorem~\ref{mainthm}, we can assume that $G$ has no bridges, as otherwise, the minimum modulus of a reliability root must be larger than $1/(n-1)$. One can characterize all bridgeless graphs $G$ of corank $2$ as follows. As $G$ has no bridges, every vertex has degree at least $2$. If we have a vertex $x$ of degree $2$, with neighbors $y$ and $z$, we remove $x$ and add in an edge from $y$ to $z$; this operation deletes a vertex and an edge, and hence leaves the corank the same. We repeat this procedure until we can no longer do so, to arrive at a graph $G^\prime$ (possibly with loops and/or multiple edges), of corank $2$, where each vertex has degree at least $3$ (in general, we can do this with any fixed corank to derive a finite list of graphs for which every graph of that corank is a {\em subdivision} of one of these graphs). If $G^\prime$ has order $n^\prime$ and size $m^\prime$, as every vertex has degree at least $3$ and the sum of the vertices is twice the number of edges, we have $2m = 2(n^\prime + 1) \geq 3n^\prime$, which implies that $n^\prime \leq 2$. The only graphs $G^\prime$ of order at most $2$ with corank $2$ and all vertices of degree at least $3$ are (i) two loops joined at a vertex, or (ii) two vertices joined by $3$ edges. This implies that $G$ must either be (i) two cycles joined at a vertex, or (ii) a {\em theta graph} consisting of two vertices $x$ and $y$ joined by three internally disjoint paths, say of lengths $l_{1}$, $l_{2}$ and $l_{3}$, each of cardinality at least $1$.

In case (i), we see, as before, that in fact there are no roots of modulus $1/(n-1)$, so we are left only with case (ii). As there are no bridges, $H_{1} = n-1$. As the only subsets of two edges whose removal leaves $G$ disconnected are two edges in one of the three internally disjoint paths, we see that 
\[ F_{2} = {{n+1} \choose {2}} - {{l_1} \choose {2}} - {{l_2} \choose {2}} - {{l_3} \choose {2}}.\]
It follows (by considering an interval partition of $\mbox{Cog}(G)$) that 
\begin{eqnarray*}
H_2 &=& {n+1 \choose 2 } - {l_1 \choose 2} - {l_2 \choose 2} - {l_3 \choose 2} - (n-1) - 2\\
&=& {n \choose 2} - {l_1 \choose 2} - {l_2 \choose 2} - {l_3 \choose 2} - 1
\end{eqnarray*}
It follows that $H_2 \leq {n \choose 2} - 1$.
Now since the corank is $2$, the $h$-polynomial will be
$$h(G,q) = H_2q^2 + (n-1)q + 1.$$
By the quadratic formula, the roots of this are
\begin{eqnarray*}
&& \frac{-(n-1) \pm \sqrt{(n-1)^2 - 4H_2}}{2H_2}. \label{quadroots}
\end{eqnarray*}
We have two cases, depending on whether the roots are real or not. 

First, if the roots are real, then $(n-1)^2 - 4H_2 \geq 0 $, that is,  $H_2 \leq (n-1)^2/4$.  The root of smallest modulus is $\frac{-(n-1) + \sqrt{(n-1)^2 - 4H_2}}{2H_2}$, and this has modulus greater than $1/(n-1)$ if and only if 
$$\frac{-(n-1) + \sqrt{(n-1)^2 - 4H_2}}{2H_2} < \frac{-1}{n-1}.$$
This holds if and only if $H_2^2 + (n-1)H_2 > 0$, which is clearly true as $H_{2} > 0 $. Secondly, if the root of (\ref{quadroots}) are nonreal, then $H_2 > (n-1)^2/4$. Both roots have the same modulus, and
$$\frac{\sqrt{(n-1)^2 + 4H_2-(n-1)^2}}{2H_2} > \frac{1}{n-1}$$
is true provided that
$$H_2 < (n-1)^2.$$
However, from above, $H_2 \leq {n \choose 2}-1$, and ${n \choose 2}-1 < (n-1)^2$ as $n \geq 3$.  Thus in this case there is no root of modulus $1/(n-1)$.

Thus, in conclusion, for connected graphs of order at least $3$, the minimum modulus of a reliability root  is $1/(n-1)$, and only occurs for a cycle of order $n$. Moreover, the only reliability root of this modulus is $-1/(n-1)$.
\end{proof}

\begin{corollary}
If $n \geq 3$, the minimum modulus $m_n$ of a reliability root of a graph of order $n$ is $1/(n-1)$, and cycle $C_n$ is the only graph of that order that has a root of this modulus, and it has only one reliability root, $-1/(n-1)$, of this modulus. If $n = 2$, then the reliability roots of smallest modulus are all the $n$--th roots of unity. \null\hfill$\Box$\par\medskip
\end{corollary}

\section{Concluding Remarks}
    
It is interesting that for $n \geq 3$, the reliability roots of smallest modulus are rational (and for $n = 2$, there is a reliability root of smallest modulus that is rational). This begs the question -- what can one say about the rational reliability roots? Surprisingly, we can give a complete answer.

\begin{theorem}
The rational reliability roots of graphs of order $n \geq 2$ are $\{1\} \cup \{-1/k: 1 \leq k \leq n-1\}$.
\end{theorem}
\begin{proof}
Let $G$ be a graph of order $n \geq 2$ with a rational root $r = a/b \neq 1$ ($1$ is always a reliability root of a connected graph of order at least $2$ as the graph fails when all edges fail). From (\ref{Hform}),
$$h(G,r) = H_dr^d + H_{d-1}r^{d-1} + \dots + H_1r + H_0.$$
By the well-known Rational Roots Theorem,  $a|H_{0}$ and $b|H_{d}$. From the facts that $H_{0} = 1$ and $h(G,q)$ clearly has no positive roots (as the coefficients are all positive), we conclude that any rational root of $h(G,q)$, and hence any rational reliability root of $G$, different from $1$, is of the form $-1/k$ where $k$ is a positive integer. Now from Theorem ~\ref{mainthm}, $k \leq n-1$, so that $r \in \{-1/k: 1 \leq k \leq n-1\}$.

To conclude our argument, note that for any $k \in \{1,2,\ldots,n-1\}$, if we take the cycle $C_{k+1}$ and recursively attach a new vertex to one existing vertex until we reach $n$ vertices in total, then the graph $G_{k,n}$ has order $n$ and rational reliability roots 1 and $-1/k$. It follows that the rational reliability roots of connected graphs of order $n \geq 2$ are precisely $\{1\} \cup \{-1/k: 1 \leq k \leq n-1\}$.
\end{proof}

What is not so clear is what occurs if one introduces some conditions on the class of graphs. For example, if one restricts to {\em simple} graphs (that is, graphs without loops and multiple edges), from above clearly the rational reliability roots of such graphs are contained in $\{1\} \cup \{-1/k: 2 \leq k \leq n-1\}$, as one can take, for $3 \leq k \leq n$, a cycle $C_{k}$ and attach leaves until one reaches order $n$. However, can $-1$ be a reliability root of a connected simple graph? We suspect not, and computation shows that this is true, at least for small $n$.

The construction of graphs of order $n$ with rational reliability roots in $\{-1/k: 2 \leq k \leq n-2\}$ require the introduction of bridges to the graph, and hence the examples are not $2$-edge-connected. What about if we restrict to $2$-edge-connected graphs? If we allow multiple edges, we can still attain the same rational roots by attaching new vertices not by a single edge but by a bundle of at least two edges. However, what about if we insist on {\em simple} $2$-edge-connected graphs? In this case, there may be rationals missing from the reliability root set. For example, among all such graphs of order $8$, the rational reliability roots are $1,-1/2,-1/3,-1/4,-1/5$ and $-1/7$.

The question is even more interesting for $2$-connected graphs, that is, those without cut vertices. We do not know whether all of $-1,-1/2,\ldots,$ $ -1/(n-1)$ can be roots. Among simple such graphs of order $n$, the rational reliability roots may be even sparser -- for order $8$, the rational roots are only $1, -1/2,-1/3.-1/4$ and $-1/7$. How to characterize which rationals can be reliability roots of $2$ -connected graphs remains an open problem.  


\section*{Acknowledgements}
 
Research of J.I.\ Brown is partially supported by grant RGPIN-2018-05227 from Natural Sciences and Engineering Research Council of Canada (NSERC).  

\singlespacing

\end{document}